\documentclass[11pt,a4paper,english,reqno]{amsart}
\usepackage[a4paper,footskip=1.5em,bottom=38mm,margin=2.20cm]{geometry}
\usepackage{amsmath,amssymb,amsthm,mathtools}
\newtheorem{theorem}{Theorem}

\newtheorem{claim}[theorem]{Claim}
\newtheorem{corollary}[theorem]{Corollary}
\usepackage{amssymb,amsmath}

\title{Tur\'an-type results for intersection graphs of boxes}
\author{Istv\'an Tomon}
\address{ETH Zurich \and MIPT Moscow}
\email{istvan.tomon@math.ethz.ch}
\author{Dmitriy Zakharov}
\address{MIPT Moscow \and HSE Moscow}
\email{s18b1\_zakharov@179.ru, zakharov2k@gmail.com}

\date{}

\begin{document}

\maketitle

\begin{abstract}
In this short note, we prove the following analog of the K\H{o}v\'ari-S\'os-Tur\'an theorem for intersection graphs of boxes. If $G$ is the intersection graph of $n$ axis-parallel boxes in $\mathbb{R}^{d}$ such that $G$ contains no copy of $K_{t,t}$, then $G$ has at most $ctn(\log n)^{2d+3}$ edges, where $c=c(d)>0$ only depends on $d$. Our proof is based on exploring connections between boxicity, separation dimension and poset dimension.

Using this approach, we also show that a construction of Basit et al. of $K_{2,2}$-free incidence graphs of points and rectangles in the plane can be used to disprove a conjecture of Alon et al. We show that there exist graphs of separation dimension 4 having superlinear number of edges.
\end{abstract}

\section{Introduction}

The celebrated K\H{o}v\'ari-S\'os-Tur\'an theorem \cite{KST54} states that if $G$ is a graph on $n$ vertices containing no copy of $K_{t,t}$, then $G$ has at most $O(n^{2-1/t})$ edges. In the past few decades, a great amount of research was dedicated to showing that this bound can be significantly improved in certain restricted families of graphs, many of which are geometric in nature. See e.g. \cite{FPSSZ17} for semi-algebraic graphs, \cite{JP20} for graphs of bounded VC-dimension, and \cite{FP08} for intersection graphs of connected sets in the plane. In particular, Fox and Pach \cite{FP08} proved that if $G$ is the intersection graph of $n$ arcwise connected sets in the plane, and $G$ contains no $K_{t,t}$, then $G$ has at most $cn$ edges, where $c=c(t)>0$ depends only on $t$. In this paper, we are interested in the question that in what meaningful ways can this result be extended in higher dimensions. That is, for which families of geometric objects is it true that if their intersection graph $G$ is $K_{t,t}$-free, then $G$ has at most linear, or almost linear number of edges? 

It turns out that already in dimension 3, one must put heavy restrictions on the family for this to hold. As a counterexample to many natural candidates, there exists a family of $n$ lines in $\mathbb{R}^{3}$, whose intersection graph is $K_{2,2}$-free and contains $\Omega(n^{4/3})$ edges. To see this, consider a configuration of $n/2$ points and $n/2$ lines on the plane with $\Omega(n^{4/3})$ incidences, which is the most possible number of incidences by the well known Szemer\'edi-Trotter theorem \cite{SzT83}. To get an intersection graph in $\mathbb{R}^{3}$, replace each point with a line parallel to the $z$-axis containing the point, and replace each line $l$ with a line $l'$ such that the projection of $l'$ to the $xy$-plane is $l$, and the lines $l'$ are pairwise disjoint. This family of $n$ lines in $\mathbb{R}^{3}$ contains no $K_{2,2}$, and has $\Omega(n^{4/3})$ intersections.

One natural family of geometric objects for which the above question becomes interesting is the family of axis-parallel boxes. In this case, we are able to prove an almost linear upper bound on the number of edges.

\begin{theorem}\label{thm:main}
Let $d,t$ be positive integers, then there exists $c=c(d)>0$ such that the following holds. If $G$ is the intersection graph of $n$ $d$-dimensional axis-parallel boxes such that $G$ contains no $K_{t,t}$, then $G$ has at most $ctn(\log n)^{2d+3}$ edges.
\end{theorem}

One might conjecture that the almost linear upper bound in Theorem \ref{thm:main} can be replaced with a linear one. This is true in case $d=2$ by the above mentioned result of Fox an Pach \cite{FP08}. However, much to our surprise,  a construction of Basit et al. \cite{BCSTT20} implies that this is not true for $d\geq 3$.

\begin{theorem}[Basit et al. \cite{BCSTT20}]\label{thm:tao}
For every $n$, there exists a $K_{2,2}$-free incidence graph of $n$ points and $n$ rectangles in the plane with $\Omega(n\frac{\log n}{\log\log n})$ edges.
\end{theorem}

 One can easily turn this construction into a bipartite intersection graph of boxes in $\mathbb{R}^{3}$. Replace each point $p$ with a box $B\times [0,1]$, where $B$ is a very small square containing $p$. Also, replace each rectangle $R$ with a box $R\times [x,x+\frac{1}{2n}]$, where $0\leq x<1$, and the intervals $[x,x+\frac{1}{2n}]$ are pairwise disjoint. The intersection graph of these $2n$ boxes is the same as the incidence graph of points and rectangles. Therefore, we get the following immediate corollary.
 
\begin{corollary}\label{cor:tao}
For every $n$, there exists a bipartite $K_{2,2}$-free intersection graph of $n$ boxes in $\mathbb{R}^{3}$ with $\Omega(n\frac{\log n}{\log\log n})$ edges.
\end{corollary} 

Our Theorem \ref{thm:main} is almost identical to one of the main results in \cite{BCSTT20}, which was done independently from us. However, the proof we present here is significantly shorter, and is based on exploring connections with other problems in graph theory. Due to our approach, we found out that the construction given by Theorem \ref{thm:tao} can be also used to give a counterexample to a conjecture of Alon et al. \cite{ABCMR18} about the number of edges in a graph of bounded separation dimension.

The \emph{separation dimension} of a graph $G$ is the smallest $d$ for which there exists an embedding $\phi:G\rightarrow \mathbb{R}^{d}$ such that if $\{x,y\}$ and $\{x',y'\}$ are disjoint edges of $G$, then the axis-parallel box spanned by $\phi(x)$ and $\phi(y)$ is disjoint from the axis parallel box spanned by $\phi(x')$ and $\phi(y')$. Alon et al. \cite{ABCMR18} conjectured that for every $d$ there exists a constant $c>0$ such that if $G$ is a graph on $n$ vertices with separation dimension $d$, then $G$ has at most $cn$ edges. They proved this in the case $d=2$. Also, Scott and Wood \cite{SW18} confirmed the conjecture for $d=3$, which also implies the bound $O(n (\log n)^{d-3})$ for $d>3$. However, we show that the conjecture no longer holds for $d\geq 6$.

\begin{theorem}\label{thm:construction}
For every $n$, there exists a graph $G$ on $n$ vertices with $\Omega(n\frac{\log n}{\log \log n})$ edges such that the separation dimension of $G$ is at most 4.
\end{theorem}

Note that dimension 4 cannot be lowered by the aforementioned result of Scott and Wood \cite{SW18}.

The connection between separation dimension and $K_{2, 2}$-free intersection graphs of boxes also implies the following almost matching bound to Theorem \ref{thm:tao}, which improves the corresponding result in \cite{BCSTT20}.

\begin{corollary}\label{cor:K22}
 If $G$ is the incidence graph of $n$ points and $n$ rectangles in the plane, and $G$ is $K_{2,2}$-free, then $G$ has at most $O(n\log n)$ edges.
\end{corollary}

\section{Boxicity, poset dimension and separation dimension}

In order to prove Theorem \ref{thm:main} and Theorem \ref{thm:construction}, let us introduce some notation. The \emph{boxicity} of a graph $G$, denoted by $\mbox{box}(G)$ is the smallest $d$ such that $G$ can be realized as the intersection graph of $d$-dimensional boxes. 

Given a partially ordered set $P$, the dimension (Duschnik-Miller dimension) of $P$, denoted by $\mbox{dim}(P)$ is the smallest $d$ such that there exists an embedding $\phi:P\rightarrow \mathbb{R}^{d}$ satisfying that $x<_P y$ if and only if $\phi(x)_i<\phi(y)_i$ for $i=1,\dots,d$. (Here, $v_i$ is the $i$-th coordinate of $v$.)

The following connection between boxicity and poset dimension was established by Adiga, Bhowmick and Chandran \cite{ABC11}. Given a graph $G$, define the bipartite poset $(P(G),\prec)$ as follows: let the elements of $P(G)$ be $V(G)\times\{0,1\}$, and let $(u,0)\prec(v,1)$ if $u=v$ or $uv\in E(G)$.

\begin{theorem}\label{dim}\cite{ABC11}
$\frac{1}{2}\mbox{box}(G)\leq\mbox{dim}(P(G))\leq 2\mbox{box}(G)+4.$ Also, if $G$ is bipartite, and $P$ is the underlying partial order, then $\mbox{dim}(P)\leq 2\mbox{box}(G)$.
\end{theorem}

The poset $P(G)$ not only estimates the boxicity of $G$ well, it also (almost) retains the property of being $K_{t,t}$-free (when referring to a poset as a graph, we refer to its comparability graph).

\begin{claim}\label{partite}
If $G$ is $K_{t,t}$-free, then $P(G)$ has a $K_{t,t}$-free induced subgraph with at least $e(G)/2$ edges.
\end{claim}

\begin{proof}
Let $(A,B)$ be a partition of $V(G)$ such that at least half of the edges of $G$ have one endpoint in $A$ and $B$. Then the subgraph of $P(G)$ induced on $\{(a,0):a\in A\}\cup \{(b,1):b\in B\}$ is $K_{t,t}$-free. Indeed, a copy of $K_{t,t}$ in this subgraph would correspond to a copy of $K_{t,t}$ in $G$ in which one of the vertex classes is in $A$, and the other is in $B$.
\end{proof}

Given two points $x$ and $y$ in $\mathbb{R}^{d}$, let $b(x,y)$ denote the box spanned by $x$ and $y$. Let $\prec$ denote the partial ordering on $\mathbb{R}^{d}$ defined as $x\prec y$ if $x_i< y_i$ for $i=1,\dots,d$.

\begin{claim}\label{boxes}
Let $V$ be a set of points in $\mathbb{R}^{d}$ and let $P=(V,\prec)$. If $P$ does not contain $K_{t,t}$, then $P$ contains no matching $\{x^1,y^1\},\dots,\{x^t,y^t\}$ of size $t$ such that $\bigcap_{i=1}^{t} b(x^i,y^i)\neq \emptyset.$
\end{claim}

\begin{proof}
Let us assume that there exists such a matching $\{x^1,y^1\},\dots,\{x^t,y^t\}$, and without loss of generality, assume that $x^i\prec y^i$ for $i=1,\dots,t$. Let $z\in \bigcap_{i=1}^{t} b(x^i,y^i)$, then $x^i\prec z\prec y^i$. But then $x^{i}\prec z\prec y^{j}$ for all $1\leq i,j\leq t$, which means that $x^1,\dots,x^t$ and $y^1,\dots,y^t$ span $K_{t,t}$ in $P$.
\end{proof}

Note that this claim also tells us that if the poset $P$ is $K_{2,2}$-free, then its separation dimension is at most $d$, as $V$ is a suitable embedding of the vertices. Let us use this to prove Theorem \ref{thm:construction}. First, we show  a somewhat weaker result.

\begin{theorem}
For every $n$, there exists a graph $G$ on $n$ vertices with $\Omega(n\frac{\log n}{\log \log n})$ edges such that the separation dimension of $G$ is at most 6.
\end{theorem}
\begin{proof}[Proof of Theorem \ref{thm:construction}]
Let $G$ be the bipartite intersection graph of $n$ boxes in $\mathbb{R}^3$ such that $G$ contains no copy of $K_{2,2}$, and $|E(G)|=\Omega(n\frac{\log n}{\log \log n})$. Such a graph exists by Corollary \ref{thm:tao}. But then $\mbox{dim}(P)\leq 2\mbox{box}(G)=6$ by Theorem \ref{dim}, where $P$ is the underlying comparability graph of $G$. We are done as $P$ has separation dimension at most 6 as well.
\end{proof}

In order to improve the dimension from 6 to 4, we just note that if $G$ is the incidence graph given by  Theorem \ref{thm:tao} instead of the intersection graph of Corollary \ref{cor:tao}, then $P$ has dimension at most 4. The proof of this follows from a similar argument as the one in \cite{ABC11}, but for the reader's convenience, we present a short proof here as well.

\begin{claim}\label{dim4}
Let $G$ be the incidence graph of points and rectangles in the plane such that no rectangle contains another, and $G$ is $K_{2,2}$-free. Let $P$ be the underlying bipartite poset, then $\mbox{dim}(P)\leq 4$. 
\end{claim}

\begin{proof}
 We show that if $G$ is $K_{2,2}$-free, then we can assume that no rectangle contains the other. Indeed,  suppose that $R\subset Q$ for some rectangles $R$ and $Q$. Then $R$ contains at most one point as $G$ is $K_{2,2}$-free. But then we can replace $R$ with a rectangle $R'$ such that $Q$ is very thin and long, $Q$ contains only the point in $R$, and $Q$ has no containment relation with any other rectangle. But then this configuration has the same incidence graph $G$.

Let $P \subset \mathbb R^2$ be a set of points and $R$ be a set of rectangles in $\mathbb R^2$ such that no rectangle contains the other. Denote by $G$ the corresponding incidence graph. Consider the map $\phi: \mathbb R^2 \rightarrow \mathbb R^4$ defined by $(x, y) \mapsto (x, -x, y, -y)$. Given a rectangle $S = \{ (x, y)~|~a \le x\le b,~ c \le y \le d\}$ on the plane, denote by $\phi(S) \in \mathbb R^4$ the point with coordinates $(a, -b, c, -d)$. Note that a point $p$ is contained in a rectangle $S$ if and only if $\phi(S) \prec \phi(p)$. Clearly, any two points are incomparable, and as no rectangle is contained in another, no two rectangles are comparable.
\end{proof}

This finishes the proof of Theorem \ref{thm:construction}. Let us continue with the proof of Theorem \ref{thm:main}.

\begin{theorem}\label{numedges}
Let $d,t$ be positive integers, then there exists $c=c(d)$ such that the following holds. Let $V$ be a set of $n$ points in $\mathbb{R}^{d}$ and let $G$ be a graph on $V$ such that $G$ contains no matching $\{x^1,y^1\},\dots,\{x^t,y^t\}$ of size $t$ satisfying $\bigcap_{i=1}^t b(x^i,y^i)\neq \emptyset$. Then $e(G)\leq ctn(\log n)^{d-1}$. 
\end{theorem}

\begin{proof}
The statement follows from a standard divide and conquer argument.

 Let us proceed by induction on $d$. First, consider the base case $d=1$. In this case, $b(x,y)$ is an interval. It is easy to show that if the intersection graph of intervals contains no $K_{2t}$, then it is $(2t-2)$-degenerate. But then $e(G)< 2tn$.

 Now suppose that $d\geq 2$. Let $f_d(n)$ denote the minimum $m$ such that any graph $G$ with the desired properties has at most $m$ edges. We show that $f_{d}(n)\leq c_{d}tn(\log n)^{d-1}$, where $c_{d}>0$ depends only on $d$.
 
 Let $G$ be a graph with the desired properties. Without loss of generality, we can assume that no two points in $V$ are on the same axis-parallel hyperplane. Let $H$ be a $(d-1)$-dimensional hyperplane perpendicular to the last coordinate axis such that at most half of the points of $V$ are on each side of $H$. Let $A$ and $B$ be the set of points of $V$  on the two sides of $H$. Let $p(x)$ denote the projection of $x$ into $H$, and let $G'$ be the graph on vertex set $p(V)$ in which $p(x)$ and $p(y)$ are joined by an edge if $xy\in E(G)$ and $x\in A$ and $y\in B$. If $x,x'\in A$ and $y,y'\in B$, then $b(x,y)\cap b(x',y')\neq \emptyset$ if and only if $b(p(x),p(y))\cap b(p(x'),p(y'))\neq \emptyset$. Therefore, $G'$ contains no matching of size $t$ such that the boxes spanned by the edges have a nonempty intersection. Hence, we deduce that
 $$e(G)=e(G[A])+e(G[B])+e(G[A,B])\leq 2f_d(n/2)+e(G')\leq 2f_d(n/2)+f_{d-1}(n).$$
 From this, we get that $f_d(n)=O(f_{d-1}(n)\log n)=O(tn(\log n)^{d-1})$, where the last equality holds by our induction hypothesis, and the constant hidden in the $O(.)$ notation only depends on $d$.
\end{proof}

In case $t=2$, Theorem \ref{numedges} can be improved. In this case, the graph $G$ has separation dimension at most $d$, which implies that $e(G)\leq cn(\log n)^{d-3}$ by the result of Scott and Wood.

After these preparations, everything is set to prove our main theorem.

\begin{proof}[Proof of Theorem \ref{thm:main}]
Let $G$ be the intersection graph of $n$ boxes in $\mathbb{R}^{d}$, and suppose that $G$ is $K_{t,t}$-free. Then by Theorem \ref{dim} and Claim \ref{partite}, there exists a $K_{t,t}$-free poset $P$ with at least $e(G)/2$ edges, whose dimension is at most $2\mbox{box}(G)+4\leq 2d+4$. But then by Claim \ref{boxes} and Theorem \ref{numedges}, we get $e(P)<ctn(\log n)^{2d+3}$, where $c=c(d)>0$ only depends on $d$. This gives $e(G)<2ctn(\log n)^{2d+3}$.
\end{proof}

Finally, let us prove Corollary \ref{cor:K22}.

\begin{proof}[Proof of Corollary \ref{cor:K22}]
Let $G$ be the incidence graph of $n$ points and $n$ rectangles in the plane, and suppose that $G$ is $K_{2,2}$-free. Then by Claim \ref{dim4}, the underlying poset $P$ of $G$ has dimension at most $4$. But as $P$ is $K_{2,2}$-free, $P$ has separation dimension at most 4, so by the result of Scott and Wood \cite{SW18}, $P$ has at most $O(n\log n)$ edges.
\end{proof}

\section*{Acknowledgements}
We would like to thank Andrey Kupavskii, J\'anos Pach and Alexandr Polyanskii for valuable discussions.

The first author was supported by the SNSF grant 200021-175573. Also, both authors were supported by the Ministry of Educational and Science of the Russian Federation in the framework of MegaGrant no 075-15-2019-1926.

\end{document}